\documentclass[preprint,12pt]{elsarticle}

\usepackage{hyperref}
\usepackage{xcolor}
\usepackage{amssymb}
\usepackage{amsmath}
\usepackage{amsthm}

\usepackage[english]{babel}
\usepackage{graphicx,epstopdf,epsfig}
\usepackage{amsfonts,epsfig,fancyhdr,graphics}
\newtheorem{theorem}{Theorem}[section]

\newtheorem{corollary}[theorem]{Corollary}
\newtheorem{proposition}[theorem]{Proposition}
\newtheorem{definition}[theorem]{Definition}

\usepackage{xcolor}

\journal{Linear Algebra and its Applications}

\begin{document}

\begin{frontmatter}

%% Title, authors and addresses

%% use the tnoteref command within \title for footnotes;
%% use the tnotetext command for theassociated footnote;
%% use the fnref command within \author or \affiliation for footnotes;
%% use the fntext command for theassociated footnote;
%% use the corref command within \author for corresponding author footnotes;
%% use the cortext command for theassociated footnote;
%% use the ead command for the email address,
%% and the form \ead[url] for the home page:
%% \title{Title\tnoteref{label1}}
%% \tnotetext[label1]{}
%% \author{Name\corref{cor1}\fnref{label2}}
%% \ead{email address}
%% \ead[url]{home page}
%% \fntext[label2]{}
%% \cortext[cor1]{}
%% \affiliation{organization={},
%%             addressline={},
%%             city={},
%%             postcode={},
%%             state={},
%%             country={}}
%% \fntext[label3]{}

\title{Schur-Horn theorem and Ky Fan's minimum principle for symplectic eigenvalues}

%% use optional labels to link authors explicitly to addresses:
%% \author[label1,label2]{}
%% \affiliation[label1]{organization={},
%%             addressline={},
%%             city={},
%%             postcode={},
%%             state={},
%%             country={}}
%%
%% \affiliation[label2]{organization={},
%%             addressline={},
%%             city={},
%%             postcode={},
%%             state={},
%%             country={}}

\author{Kennett L. Dela Rosa, Aedan Jarrod A. Potot} %% Author name

%% Author affiliation
\affiliation{organization={Institute of Mathematics, University of the Philippines Diliman},%Department and Organization
            addressline={C.P. Garcia Ave., UP Campus}, 
            city={Quezon City},
            postcode={1101}, 
            country={Philippines}}

%% Abstract
\begin{abstract}
%% Text of abstract
%The Ky Fan principle and the Schur-Horn theorem are known to have analogues for symplectic eigenvalues. In this study, 
The symplectic analogues of Schur's theorem and Ky Fan's minimum principle are shown to be equivalent. Moreover, the symplectic Schur's and Horn's theorems are extended to generalized means.
\end{abstract}

%%Graphical abstract
%%\begin{graphicalabstract}
%%\includegraphics{grabs}
%%\end{graphicalabstract}

%%Research highlights
%%\begin{highlights}
%%\item Research highlight 1
%%\item Research highlight 2
%%\end{highlights}

%% Keywords
\begin{keyword}
  positive definite matrix \sep symplectic eigenvalue \sep Schur-Horn theorem \sep Ky Fan minimum/maximum principle
%% keywords here, in the form: keyword \sep keyword

%% PACS codes here, in the form: \PACS code \sep code

%% MSC codes here, in the form: \MSC code \sep code
%% or \MSC[2008] code \sep code (2000 is the default)
\MSC[2020] 15A42 \sep 15A99 \sep 15B57 \sep 15B99 
\end{keyword}

\end{frontmatter}

%% Add \usepackage{lineno} before \begin{document} and uncomment 
%% following line to enable line numbers
%% \linenumbers

%% main text
%%

\section{Introduction}

Given $z_1,z_2,\ldots,z_n\in \mathbb{R}$, write $z_1^\uparrow,z_2^\uparrow,\ldots, z_n^\uparrow$ to mean that the $z_j$'s are arranged as $z_1^\uparrow\leq z_2^\uparrow\leq \cdots\leq z_n^\uparrow$. Let $x=[x_j]_{j=1}^n$ and $y=[y_j] _{j=1}^n$ be real vectors. Then $x$ is \textit{weakly supermajorized} by $y$, denoted $x\prec^w y$, if
\begin{equation}
		\sum_{j=1}^k x_j^\uparrow \geq \sum_{j=1}^k y_j^\uparrow,\ \textup{for all}\ k=1,\ldots,n.\label{1introductioneq1}
\end{equation}
If \eqref{1introductioneq1} holds and $\sum_{j=1}^n x_j^\uparrow = \sum_{j=1}^n y_j^\uparrow$, then $x$ is \textit{majorized} by $y$, denoted $x\prec y$.

The Schur-Horn theorem reveals that majorization is the precise relationship between the diagonal entries and eigenvalues of a Hermitian matrix. Given Hermitian $A=[a_{ij}]_{i,j=1}^n$ with eigenvalues $\lambda_1, \lambda_2, \ldots, \lambda_n$, let $\textup{diag}(A)=[a_{jj}]_{j=1}^n$ and $\lambda(A)=[\lambda_j]_{j=1}^n$. Schur's theorem asserts that $\textup{diag}(A)\prec \lambda(A)$ \cite[Theorem 4.3.45]{matan1}. Conversely, Horn's theorem guarantees that for real vectors $x,y$ where $x\prec y$, there exists real symmetric matrix whose main diagonal entries are the entries of $x$ and whose eigenvalues are the entries of $y$ \cite[Theorem 4.3.48]{matan1}. The Ky Fan's minimum/maximum principle ensures that the minimum (resp. maximum) value of $\mbox{tr}(X^*AX)$ over all $n$-by-$k$ matrices $X$ with $X^*X=I_k$ is given by $\sum^{k}_{j=1}\lambda_{j}^\uparrow$ (resp., $\sum^{n}_{j=n-k+1}\lambda_{j}^\uparrow$) \cite[Corollary 4.3.39]{matan1}.
 %\begin{equation}\label{kyfanprinciple}
%\[\begin{array}{rcl}
%\displaystyle\min_{X \in \mathcal{U}_{k}}\mathrm{tr}(X^{*}AX)&=&\displaystyle\sum^{k}_{j=1}\lambda_{j}\\
%\displaystyle\max_{X \in \mathcal{U}_{k}}\mathrm{tr}(X^{*}AX)&=&\displaystyle\sum^{n}_{j=n-k+1}\lambda_{j}
%\end{array}\]
% \end{equation}
 %where $\mathcal{U}_{k}=\left\{  X \in M_{n,k}\left( \mathbb{C} \right) \mid X^{*}X=I_{k}  \right\}$.
Schur's theorem is known to be equivalent to Ky Fan's minimum/maximum principle \cite[Exercises II.1.12-13]{matan2}. 

Let $A$ be $2n$-by-$2n$ real positive definite. Williamson's theorem implies that $A=W (D \oplus D)W^\top$ where $W$ is symplectic and $D$ is diagonal \cite{sympeigenvalues,williamsons}. The diagonal entries of $D$, say $\delta_1,\ldots,\delta_n$, are known as the \textit{symplectic eigenvalues} of $A$. There has been great interest in exploring symplectic eigenvalue analogues of classical results on Hermitian matrices due to its applications in mechanics \cite{classicalmech,sympgeomandquantmech,sympgroups,symptopology} and quantum information \cite{gaussianquant, gaussianchannel, quantummarginal, gaussianstates}. Theoretical questions have been considered such as the precise conditions for majorization in the symplectic Schur's theorem  \cite{sympeigenvalues2}, the symplectic Weyl's inequalities \cite{sympeigenvalues3}, and sums/products of symplectic eigenvalues \cite{sumsympeigenvalues}. The symplectic Schur's theorem asserts that the vector of symplectic eigenvalues of $A$ supermajorizes its vector of symplectic diagonals associated with the geometric and arithmetic means \cite[Theorems 1 and 3]{sympschurhorn}. For symplectic diagonals using the geometric and arithmetic means, the converse holds and is known as the symplectic Horn's theorem \cite[Theorems 1 and 3]{sympschurhorn}. On the other hand, the symplectic Ky Fan's minimum principle implies that the minimum value of $\mathrm{tr}(X^\top AX)$ over all $2n$-by-$2k$ matrices $X$ with $X^{\top} J_{2n} X=J_{2k}$ is equal to $2\sum^{k}_{j=1}\delta_{j}^\uparrow$ \cite[Theorem 5]{sympeigenvalues}.
%\[\displaystyle\min_{X \in \mathcal{S}_{k}}\mathrm{tr}(X^\top AX)=2\displaystyle\sum^{k}_{j=1}d_{j}\]
%where $\mathcal{S}_{k}=\left\{  X \in M_{2n,2k}(\mathbb{R}): X^{\top}J_{2n}X=J_{2k}  \right\}$.

%Let $d_{s}(A)$ denote the vector of symplectic eigenvalues of $A$. In \cite{sympschurhorn}, Bhatia and Jain considered two \textit{symplectic diagonals}, namely $\Delta_s(A)=[(\Delta(A_{11})_i \Delta(A_{22})_i)^{1/2}]$ and $\Delta_c(A)=\left[\dfrac{\Delta(A_{11})_{i}+\Delta(A_{22})_{i}}{2}\right]$. Given these symplectic diagonals, they proved the Schur-Horn theorem for symplectic eigenvalues. That is, if $A \in M_{2n}\left( \mathbb{R} \right)$ is positive definite, then $\Delta_{s}(A)\prec ^{w}d_{s}(A)$ and $\Delta_{c}(A)\prec ^{w}d_{s}(A)$. Conversely, if $x,y \in \mathbb{R}^{n}$ such that $x\prec^{w} y$, then there exists positive definite $A \in M_{2n}(\mathbb{R})$ such that $x=\Delta_{s}(A)$ and $y=d_{s}(A)$ or $x=\Delta_c(A)$ and $y=d_{s}(A)$. In \cite{sympeigenvalues}, they proved the Ky Fan principle for symplectic eigenvalues. Given positive definite $A \in M_{2n}(\mathbb{R})$, the symplectic Ky Fan principle states that for all $k=1,\dots,n$,
%$$
%\min_{X \in \mathcal{S}_{k}}\mathrm{tr}\left(X^{\top}AX\right)=2\sum^{k}_{j=1}d_s(A)_{j}^{\uparrow}
%$$
%where $\mathcal{S}_{k}=\left\{  X \in M_{2n,2k}(\mathbb{R}): X^{\top}J_{2n}X=J_{2k}  \right\}$.

The goal of this paper is to extend the main results in \cite{sympschurhorn} for generalized means. In particular, we prove that the symplectic Horn's theorem holds \ref{theorem6.3}) for any mean while the symplectic Schur's theorem holds under a mild condition on the mean (Theorem \ref{theorem6.2}). As a consequence, we provide a unified proof of Theorems 1 and 3 in \cite{sympschurhorn}. Moreover, we show that the symplectic Schur's theorem is equivalent to the symplectic Ky Fan minimum principle (Theorem \ref{equiv}). 

\section{Preliminaries}

Let $M_{m,n}$ be the set of all $m$-by-$n$ real matrices. Set $M_n:=M_{n,n}$ and $\mathbb{R}^n:=M_{n,1}$. Denote by $I_n$ the $n$-by-$n$ identity matrix and $0_n$ the $n$-by-$n$ zero matrix. Let $\mathrm{diag}(d_{1},d_2,\dots,d_{n})$ be the $n$-by-$n$ diagonal matrix with diagonal entries $d_{1},d_2,\dots,d_{n}$. Given $A=[a_{ij}]_{i,j=1}^n$, let $\textup{diag}(A)=[a_{jj}]_{j=1}^n$ and $\sigma(A)$ be the set of all eigenvalues of $A$. If $A_j \in M_{m_j,n_j}$ for $j=1,\ldots,\ell $, denote their direct sum by $\oplus^\ell _{j=1} A_{j}$.

\begin{definition}
 Let $A_{j}=\begin{bmatrix}P_{j} & Q_{j} \\ R_{j} & S_{j}\end{bmatrix}$ where $P_j,S_j\in M_{m_j}$ for $j=1,\ldots,\ell $. The \textit{expanding sum of $A_{1},\dots,A_{s}$} is defined as
\[\boxplus^s_{j=1} A_{j}=\begin{bmatrix}
\oplus^\ell _{j=1} P_{j} & \oplus^\ell _{j=1} Q_{j} \\
\oplus^\ell _{j=1} R_{j} & \oplus^\ell _{j=1} S_{j}
\end{bmatrix}.\]
\end{definition}

The expanding sum is also known as the $s$\textit{-direct sum} (see \cite{sympeigenvalues}). Note that $\boxplus^\ell _{j=1} A_{j}$ is similar to $\oplus^\ell _{j=1} A_{j}$, and hence their eigenvalues are the same. Moreover, $\left(\boxplus^\ell _{j=1} A_{j}\right)\left(\boxplus^\ell _{j=1} B_{j}\right)=\boxplus^\ell _{j=1} A_{j}B_{j}$.

\begin{definition}
    Let $m_{1},\dots,m_{k}\in \mathbb{N}$ such that $ \sum^{k}_{j=1}m_{j}=n$. Let $A=\begin{bmatrix}[P_{ij}]_{i,j=1}^k & [Q_{ij}]_{i,j=1}^k \\ [R_{ij}]_{i,j=1}^k& [S_{ij}]_{i,j=1}^k\end{bmatrix}$ where $P_{jj},Q_{jj},R_{jj},S_{jj} \in M_{m_{j}}$ where $j=1,\ldots,k$. The \textit{$s$-pinching} of $A$ relative to $\left( m_{1},\dots,m_{k} \right)$ is
$
\mathcal{C}^{s}(A)=\boxplus ^{k}_{j=1}\begin{bmatrix}
P_{jj} & Q_{jj} \\
R_{jj} & S_{jj}
\end{bmatrix}.
$
\end{definition}

Note that if $A\in M_{2n}$ is positive definite, then so is $\mathcal{C}^s(A)$. Define $J_{2n}=\begin{bmatrix}
    0_n & I_n \\ - I_n & 0_n
\end{bmatrix}$. Observe that $J_{2m+2n}=J_{2m}\boxplus J_{2n}$, and so $J_{2n}=\boxplus_{j=1}^nJ_2$.

\begin{definition}
A matrix $W\in M_{2n}$ is said to be \textit{symplectic} if $W^\top J_{2n} W=J_{2n}$.
\end{definition}

%Observe that $W\in M_{2n}$ is symplectic if and only if $W$

\begin{proposition}\label{kd_prop2.4}
Let $U\in M_{2m}$ and $V,W\in M_{2n}$ be symplectic. Then the following are also symplectic: $W^\top$, $W^{-1}$, $VW$, and $U\boxplus V$.
\end{proposition}
\begin{proof}
Since $W^\top J_{2n}W=J_{2n}$, $W$ is nonsingular and so \[J_{2n}=(W^\top)^{-1}J_{2n}W^{-1}=(W^{-1})^{\top}J_{2n}W^{-1}.\] Moreover, 
\[-J_{2n}=J_{2n}^{-1}=[(W^\top)^{-1}J_{2n}W^{-1}]^{-1}=WJ_{2n}^{-1}W^{\top}=-(W^\top)^\top J_{2n}W^{\top}.\]
Hence, $W^{-1}$ and $W^\top $ are symplectic. Now, $VW$ is symplectic since
\[(VW)^\top J_{2n} VW=W^\top V^\top J_{2n} VW=W^\top J_{2n} W=J_{2n}.\] Finally, $U\boxplus V$ is symplectic due to
\[\begin{array}{rcl}
(U\boxplus V)^\top J_{2m+2n} (U\boxplus V)&=&(U^\top \boxplus V^\top )(J_{2m}\boxplus J_{2n} )(U\boxplus V)\\
&=&(U^\top J_{2m}U)\boxplus (V^\top J_{2n}V)\\
&=&J_{2m}\boxplus J_{2n}\\
&=&J_{2m+2n}.
\end{array}\]
\end{proof}

%Let $A\in M_n(\mathbb{R})$ be nonsingular. Then $A\oplus \left(A^{-1}\right)^{\top}$ is symplectic by the previous proposition.

%We observe a property of a matrix whose columns are derived from the columns of a symplectic matrix.
\begin{proposition}
\label{proposition3.1}
Let $W=\begin{bmatrix}
P & Q \\
R & S
\end{bmatrix}\in M_{2n}$ where $P,S \in M_{n}$. 
Then $W$ is symplectic if and only if $P^{\top}S-R^{\top}Q=I_n$ and $P^{\top}R$ and $Q^{\top}S$ are symmetric.
\end{proposition}

\begin{proof}
Observe that
\begin{equation*}
W^{\top}J_{2n}W =\begin{bmatrix}
P^{\top}R-R^{\top}P & P^{\top}S-R^{\top}Q \\
Q^{\top}R-S^{\top}P & Q^{\top}S-S^{\top}Q 
\end{bmatrix}.\label{symplecticeq1}
\end{equation*}
The assertion follows by comparing the above with the block components of $J_{2n}$.
\end{proof}

\begin{proposition}
\label{proposition3.2}
Let $X=[X_1\ X_2]$ where $X_1,X_2\in M_{2n,k}$. Then $X^\top J_{2n} X=J_{2k}$ if and only if there exist $Y_1,Y_2\in M_{2n,n-k}$ such that $W=[X_1\ Y_1\ X_2\ Y_2]$ is symplectic.
\end{proposition}

\begin{proof}
The backward implication holds since
\begin{equation*}
J_{2k}\boxplus J_{2n-2k}=J_{2n}=W^{\top}J_{2n}W =\begin{bmatrix}
X_{1}^{\top}J_{2n}X_{1} & *& X_{1}^{\top}J_{2n}X_{2}&* \\
* & * & * & * \\
X_{2}^{\top}J_{2n}X_{1} & * & X_{2}^{\top}J_{2n}X_{2} & * \\
* & * & * & *
\end{bmatrix}.
\end{equation*}
The forward implication follows from the symplectic analogue of the Gram-Schmidt process (see \cite[Theorem 1.1]{sympgeom} or \cite[Theorem 1.15]{sympgeomandquantmech}).
\end{proof}

The geometric mean and arithmetic mean are basic examples of means; such a notion is generalized in the following way \cite{means, posdef}.

\begin{definition}
    A \textit{(generalized) mean} $\mathbf M$ is a positive continuous function on $\left( 0,\infty \right) \times \left( 0,\infty \right)$ that satisfies the following conditions for all $a,b>0$:
\begin{enumerate}
	\item[(i)] $\mathbf M(a,b)=\mathbf M(b,a)$;
	\item[(ii)] $\mathbf M\left( r a,r b \right)=r \mathbf M(a,b)$ for all $r>0$;
	\item[(iii)] $\mathbf M(a,b)$ is monotone increasing in $a$ and $b$;
	\item[(iv)] If $a\leq b$, then $a\leq \mathbf M(a,b)\leq b$.
\end{enumerate}
\end{definition}

Observe that $\mathbf{M}(a,a)=a$ for all $a>0$. In \cite{sympschurhorn}, the authors considered symplectic diagonals corresponding to the geometric mean and arithmetic mean. We generalize this notion, and define an associated symplectic diagonal for a given mean.

\begin{definition}
Let $A=[a_{ij}]_{i,j=1}^{2n} \in M_{2n}$ where $a_{ii}>0$. The \textit{symplectic diagonal of} $A$ \textit{associated with a given mean} $\mathbf{M}$ is $\textup{diag}_{\mathbf M}(A)=\begin{bmatrix}\mathbf M\left( a_{i}, a_{n+i} \right)\end{bmatrix}_{i=1}^n.$
\end{definition}

\section{Symplectic eigenvalues}
%In \cite{williamsons}, Williamson showed that if $A \in M_{2n}\left( \mathbb{R} \right)$ is positive definite, then there exists symplectic $U\in M_{2n}(\mathbb{R})$ such that $U^{\top}AU=D\oplus D$,	where $D\in M_n(\mathbb{R})$ is a positive definite diagonal matrix. The entries of $D$ are called the \textit{symplectic eigenvalues} of $A$.
Let $A\in M_{2n}$ be positive definite. By Williamson's theorem \cite{sympeigenvalues,williamsons}, $A=W(D \oplus D)W^\top $ where $W\in M_{2n}$ is symplectic and $D=\textup{diag}(\delta_1,\ldots, \delta_n)$. The diagonal matrix $D$ is unique up to permutation of the diagonal entries. Now, Proposition \ref{proposition3.1} implies that $P\oplus P$ is symplectic for any permutation matrix $P$, and so it may be assumed that $A=\widetilde{W} (\widetilde{D}\oplus \widetilde{D} )\widetilde{W}^\top$ where $\widetilde{W}\in M_{2n}$ is symplectic and $\widetilde{D}=\textup{diag}(\delta_1^\uparrow,\ldots,\delta_n^\uparrow).$ The diagonal entries $\delta_1,\ldots,\delta_n$, are known as the \textit{symplectic eigenvalues} of $A$. In this case, denote by $\delta(A)=[\delta_j^\uparrow]_{j=1}^n$.

As observed in \cite{sympschurhorn,sympeigenvalues}, the next result gives a characterization of symplectic eigenvalues.

\begin{proposition}
\label{proposition4.3}
Let $A\in M_{2n}$ be positive definite with symplectic eigenvalues $\delta_1,\ldots,\delta_n$. Then $\{|\lambda|: \lambda\in \sigma(J_{2n}A)\}=\{|\lambda|: \lambda\in \sigma(A^{\frac12}J_{2n}A^{\frac12})\}=\{\delta_1,\ldots,\delta_n\}.$

%The modulus of eigenvalues of $J_{2n}A$ or $A^{1/2}J_{2n}A^{1/2}$ are the symplectic eigenvalues of $A$.
\end{proposition}

\begin{proof}
Observe that
\begin{align*}
\det\left( \lambda I_{n}-J_{2n}A \right)  & =\det\left( A^{-1/2}\left( \lambda I_{n}-A^{1/2}J_{2n}A^{1/2} \right)A^{1/2} \right) \\
 &=\det \left( \lambda I_{n}-A^{1/2}J_{2n}A^{1/2} \right).
\end{align*}
Hence, $\sigma(J_{2n}A)=\sigma(A^{1/2}J_{2n}A^{1/2})$, and it suffices to prove the assertion for $J_{2n}A$. By \cite{sympeigenvalues,williamsons}, $W^{\top}AW=D\oplus D$ where $W\in M_{2n}$ is symplectic and $D=\mathrm{diag}(\delta_1,\dots,\delta_n)$. Now, Proposition \ref{kd_prop2.4} guarantees that
\[
     W^{-1}J_{2n}AW =W^{-1}J_{2n}(W^{-1})^\top W^\top AW = J_{2n}(D\oplus D) =\boxplus^{n}_{j=1}\delta_jJ_2,\]
and so $J_{2n}A$ is similar to $\bigoplus^n_{j=1}\delta_j J_2$. It follows that $\sigma(J_{2n} A)=\{\pm i\delta_1,\ldots,\pm i\delta_n\}$. The assertion follows by taking the modulus of elements of $\sigma(J_{2n}A)$.
\end{proof}

%Symplectically congruent matrices have the same symplectic eigenvalues.

\begin{proposition}
\label{proposition4.4}
Let $A \in M_{2n}$ be positive definite. For any symplectic $W \in M_{2n}$, $A$ and $W^{\top}A W$ have the same symplectic eigenvalues.
\end{proposition}

\begin{proof} By the similarity
\begin{equation*}
J_{2n}(W^{\top}AW) =J_{2n}W^{\top}J_{2n}^{-1}J_{2n}AW=(J_{2n}W^{\top}J_{2n}^{-1})J_{2n}AW=W^{-1}(J_{2n}A)W,
\end{equation*}
note that $\sigma(J_{2n}W^{\top}AW)=\sigma(J_{2n}A)$, and so the assertion holds by Proposition \ref{proposition4.3}.
\end{proof}

%We note that the expanding sum of positive definite matrices is positive definite and the expanding sum of symplectic matrices are symplectic.

\begin{proposition}
\label{proposition4.6}
Let $A \in M_{2m}$ and $B \in M_{2n}$ be positive definite. Then the symplectic eigenvalues of $A\boxplus B$ are the symplectic eigenvalues of $A$ and $B$.
\end{proposition}
\begin{proof}
Observe that $J_{2m+2n}(A\boxplus B)=J_{2m}A\boxplus J_{2n}B$ is similar to $J_{2m}A \oplus J_{2n}B$. Hence, $\sigma(J_{2m+2n}(A\boxplus B))=\sigma(J_{2m}A)\cup \sigma(J_{2n}B).$ The assertion follows from Proposition \ref{proposition4.3}.
\end{proof}

\section{Symplectic Schur-Horn theorem}

%In \cite{sympeigenvalues}, Bhatia and Jain proved that if $A \in M_{2n}\left( \mathbb{R} \right)$ is positive definite and $\mathcal{C}^{s}(A)$ is an $s$-pinching of $A$ relative to $\left( m_{1},\dots,m_{k} \right)$. Then $d_{s}\left( \mathcal{C}^{s}(A) \right) \prec^{w} d_{s}(A).$

%Given $x,y\in \mathbb{R}^n$, denote $x\leq y$ when $x_i\leq y_i$ for all $i=1,\dots,n$. In \cite{majorization}, it was shown that $x\prec^{w} y$ if and only if there exists $z \in \mathbb{R}^{n}$ such that $z\prec y$ and $z\leq x$. We are now ready to prove the symplectic Schur and Horn theorems for generalized means.

Given $x=[x_j],y=[y_j]\in \mathbb{R}^n$, write $x\geq y$ to mean $x_j\geq y_j$ for all $j=1,\ldots,n$. In this case, $x\prec^w y$. 

We are now ready to prove the symplectic Schur's theorem for generalized means. The argument follows the strategy in the forward implication of \cite[Theorem 1]{sympschurhorn}, extended to the general case.

\begin{theorem}
\label{theorem6.2}
Let $\mathbf M$ be a mean such that $\sqrt{ab}\leq \mathbf M(a,b)$ for all $a,b>0$. If $A\in M_{2n}$ is positive definite, then $\textup{diag}_{\mathbf M}(A) \prec^{w} \delta(A).$ 
\end{theorem}

\begin{proof}
 Write $A=\begin{bmatrix}A_{11}& A_{12} \\ A_{12}^{\top} & A_{22}\end{bmatrix}$ where $A_{11},A_{22} \in M_{n}$ and let $\mathcal{C}^s(A)$ be the $s$-pinching of $A$ relative to $(1,\dots,1)$. Then
  $
  \mathcal{C}^{s}(A)=\boxplus_{j=1}^n\begin{bmatrix}
      \alpha_{j} & \beta_{j} \\ 
      \beta_{j} & \gamma_{j}
  \end{bmatrix}
  $
  where $\alpha_{j}, \beta_j,\gamma_j$ are the main diagonal entries of $A_{11},A_{12},A_{22}$, respectively. Since $J_{2n}\mathcal{C}_{s}(A)=\boxplus_{j=1}^n\begin{bmatrix}     \beta_j   & \gamma_j\\-\alpha_j  &  -\beta_j \end{bmatrix}$, it follows that the symplectic eigenvalues of $\mathcal{C}_s(A)$ are $\sqrt{\alpha_j\gamma_j-\beta_j^2}$ for $j=1,\ldots,n$. Set $y=[\sqrt{\alpha_j\gamma_j-\beta_j^2}]_{j=1}^n$. Now, for each $j=1,\ldots, n$, 
  \[y_j=\sqrt{\alpha_j\gamma_j-\beta_j^2}\leq \sqrt{\alpha_j\gamma_j}\leq \mathbf{M}(\alpha_j,\gamma_j).\]
  It follows that $\textup{diag}_{\mathbf M}(A)\prec^w y$. By \cite[Theorem 9]{sympeigenvalues}, $y\prec^w \delta(A)$. Thus, $\textup{diag}_{\mathbf{M}}(A)\prec^w \delta(A)$ by transitivity \cite[Remark II.1.2]{matan2}.\end{proof}

The next result is the symplectic Horn's theorem for generalized means. The argument follows that of the backward implication of \cite[Theorem 1]{sympschurhorn}, but treats generalized means directly, yielding a unified proof that includes as special cases Theorems 1 and 3 in \cite{sympschurhorn}.
\begin{theorem}
\label{theorem6.3}
Let $\mathbf{M}$ be a mean. If $x,y\in \mathbb{R}^n$ have positive entries such that $x\prec ^{w}y$, then there exists positive definite $A\in M_{2n}$ such that $x=\textup{diag}_{\mathbf M}(A)$ and $y=\delta(A)$.
\end{theorem}
\begin{proof}
By \cite[Ch. 5. Sect. A.9.a]{majorization}, there exists $z\in \mathbb{R}^n$ such that $z\prec y$ and $z\leq x$. Since the entries of $y$ are all positive, it follows that the entries of $z$ are also positive. Let $x=[x_j]_{j=1}^n$, $y=[y_j]_{j=1}^n$, $z=[z_j]_{j=1}^n,$ and $Y=\textup{diag}(y_1,\ldots,y_n)$. Horn's theorem guarantees the existence of orthogonal $U\in M_n$ such that $\textup{diag}(UYU^\top)=z$. Note that $U\oplus U$ is symplectic due to Proposition \ref{proposition3.1}. Let $B=(UYU^\top) \oplus (UYU^\top )$. By Proposition \ref{proposition4.4}, $\delta(B)=\delta(Y)=y$. Note that $\textup{diag}_{\mathbf{M}}(B)=[M(z_j,z_j)]_{j=1}^n=z$. Consider the continuous function $f: SL(2,\mathbb R)\to \mathbb{R}$ defined by $f\left( \begin{bmatrix} p& q\\ r& s\end{bmatrix}\right)=\mathbf{M}(p^2+q^2,r^2+s^2)$. Note that $f(I_2)=\mathbf{M}(1,1)=1$, and since $SL(2,\mathbb R)$ is connected \cite{connectedness}, so is $f(SL(2,\mathbb R))$ by continuity of $f$. Moreover, $f(SL(2,\mathbb R))$ is unbounded because
\[f\left(\begin{bmatrix} p& 0\\ p& p^{-1}\end{bmatrix}\right)=\mathbf{M}\left(p^2,p^2+p^{-2}\right)\geq p^2 \ \textup{for any}\ p>0.\] Hence, $[1,\infty)\subseteq f(SL(2,\mathbb R))$. Now, for each $j=1,\ldots, n$, $\frac{x_j}{z_j}\in [1,\infty)$, and this implies $\frac{x_j}{z_j}=\mathbf{M}(p_j^2+q_j^2, r_j^2+s_j^2)$ for some $p_j,q_j,r_j,s_j\in \mathbb{R}$ with $p_js_j-q_jr_j=1$. Let $W=\boxplus_{j=1}^n\begin{bmatrix} p_j& q_j\\ r_j& s_j\end{bmatrix}$. By Propositions \ref{kd_prop2.4}-\ref{proposition3.1}, $W$ is symplectic due to the assumption on $p_j,q_j,r_j,s_j$. Define $A=W BW^{\top}$. Observe that $\delta(A)=\delta(B)=y$ by Proposition \ref{proposition4.4}. Moreover, \[\begin{array}{rcl}\textup{diag}_{\mathbf{M}}(A)&=&[\mathbf{M}(p_j^2z_j+q_j^2z_j,r_j^2z_j+s_j^2z_j)]_{j=1}^n\\
&=&[z_j\mathbf{M}(p_j^2+q_j^2,r_j^2+s_j^2)]_{j=1}^n\\
&=&\left[ z_j\left( \frac{x_j}{z_j}\right)\right]_{j=1}^n\\
&=&x.
\end{array}\] 
\end{proof}

%\begin{theorem}
%\label{theorem6.3}
%Let $x,y\in \mathbb{R}^n$ with positive entries such that $x\prec ^{w}y$ and $\mathbf M(u,v)$ be a mean such that $\mathbf G(u,v)\leq \mathbf M(u,v)$. Define the function $f(t)=M(t,t^{-1})$ where $t\in [1,\infty)$. Suppose $\mathrm{ran }f=[1,\infty)$. Then there exists positive definite $A\in M_{2n}(\mathbb{R})$ such that $x=\Delta_{\mathbf M}(A)$ and $y=d_{s}(A)$.
%\end{theorem}

%Theorems \ref{theorem6.2} and \ref{theorem6.3} generalize the results in \cite{sympschurhorn} \textcolor{red}{[citation]}.

\section{Symplectic Ky Fan's minimum principle}

Given $A\in M_{2n}$, denote by ${\cal S}_k(A)=\{X^\top A X: X\in M_{2n,2k}\ \textup{with}\ X^\top J_{2n} X=J_{2k}\}.$

\begin{theorem}\label{equiv}
The following are equivalent:
\begin{enumerate}[(i)]
\item For all positive definite $A\in M_{2n}$, $\textup{diag}_{\mathbf{M}}(A)\prec^w \delta(A)$;
\item For all positive definite $A\in M_{2n}$ with symplectic eigenvalues $\delta_1,\ldots, \delta_n$, $\sum_{j=1}^k\delta_j^\uparrow=\min\{\sum_{j=1}^k \mathbf{M}(b_{jj},b_{k+j,k+j}): [b_{ij}]_{i,j=1}^{2k}\in {\cal S}_k(A)\}$.
\end{enumerate}
\end{theorem}
\begin{proof}
For the forward implication, let $A\in M_{2n}$ be positive definite with symplectic eigenvalues $\delta_1,\ldots,\delta_n$. Let $[b_{ij}]_{i,j=1}^{2k}=X^\top AX$ where $X\in M_{2n,2k}$ with $X^\top J_{2n}X=J_{2k}$. Write $X=[X_1\ X_2]$ where $X_1,X_2\in M_{2n,k}$ and so Proposition \ref{proposition3.2} implies that $W=[X_1\ Y_1\ X_2\ Y_2]$ is symplectic for some $Y_1,Y_2\in M_{2n,n-k}$. Let $C=W^\top A W$, and write $z=[z_j]_{j=1}^n=\textup{diag}_{\mathbf{M}}(C)$. By assumption and Proposition \ref{proposition4.4}, $z=\textup{diag}_{\mathbf{M}}(C)\prec^w \delta(C)=\delta(A)$, and so
\[
\sum_{j=1}^k\delta_j^\uparrow \leq \sum_{j=1}^k z_j^\uparrow
\leq \sum_{j=1}^k z_j
= \sum_{j=1}^k \mathbf{M}(c_{jj},c_{n+j,n+j})
=\sum_{j=1}^k\mathbf{M}(b_{jj}, b_{k+j,k+j}).
\]
Since $X$ is arbitrary, $\sum_{j=1}^k\delta_j^\uparrow\leq \inf\{\sum_{j=1}^k \mathbf{M}(b_{jj},b_{n+j,n+j}): [b_{ij}]_{i,j=1}^{2k}\in {\cal S}_k(A)\}.$ To see that the minimum is attained, Williamson's theorem \cite{sympeigenvalues,williamsons} guarantees that $W^\top AW=D\oplus D$ where $W\in M_{2n}$ is symplectic and $D=\textup{diag}(\delta_1^\uparrow,\ldots,\delta_n^\uparrow)$. Write $W=[X_1\ Y_1\ X_2\ Y_2]$ where $X_1,X_2\in M_{2n,k}$, and consider $X^\top  AX$ where $X=[X_1\ X_2]$. Note that $X^\top AX\in {\cal S}_k(A)$ since $X^\top J_{2n}X=J_{2k}$ due to Proposition \ref{proposition3.2}. Moreover, if $[b_{ij}]_{i,j=1}^{2k}=X^\top AX $, then \[\sum_{j=1}^k \mathbf{M}(b_{jj},b_{k+j,k+j})=\sum_{j=1}^k\mathbf{M}(\delta_j^\uparrow,\delta_j^\uparrow)=\sum_{j=1}^k\delta_j^\uparrow.\] This proves the forward implication. 
 
For the backward implication, let $A=[a_{ij}]_{i,j=1}^{2n}\in M_{2n}$ be positive definite with symplectic eigenvalues $\delta_1,\ldots,\delta_n$. Write $z=[z_j]_{j=1}^n=\textup{diag}_{\mathbf{M}}(A)$. Let $e_1,\ldots, e_n\in\mathbb{R}^n$ denote the standard basis vectors of $\mathbb{R}^n$. We define $e_{i_1},e_{i_2},\ldots, e_{i_n}\in \mathbb{R}^n$ in such a way that $z_1^\uparrow=z_{i_1},z_2^\uparrow=z_{i_2}, \ldots,z_n^\uparrow =z_{i_n}$. In particular, the indices are further described as follows. If $z_j^\uparrow\neq z_\ell^\uparrow$, then there exist distinct $i_j, i_\ell$ such that $z_j^\uparrow=z_{i_j}$ and $z_\ell^\uparrow= z_{i_\ell}$. Suppose $z_j^\uparrow=z_{j+1}^\uparrow=\cdots=z_{j+r}^\uparrow$ are all the possible repetitions of $z_j^\uparrow $ in $z$. There exists unique $(r+1)$-tuple $(i_j,i_{j+1},\ldots,i_{j+r})$ with $i_j<i_{j+1}<\cdots<i_{j+r}$ such that $z_j^\uparrow=z_{i_j}=z_{i_{j+1}}=\cdots=z_{i_{j+r}}$. Now, let $k=1,\ldots,n$, and consider $P=[e_{i_1}\ e_{i_2}\ \cdots\ e_{i_k}]\in M_{n,k}$. Let $X=P\oplus P$, and note that $X^\top J_{2n}X=J_{2k}$ since $P^\top P=I_k$. Then $X^\top AX\in {\cal S}_k(A)$, and its main diagonal entries are $a_{i_1,i_1},\ldots,a_{i_k,i_k},a_{k+i_1,k+i_1},\ldots,a_{k+i_k,k+i_k}$. By assumption, 
\[\sum_{j=1}^k\delta_j^\uparrow\leq\sum_{j=1}^k \mathbf{M}(a_{i_j,i_j},a_{k+i_j,k+i_j})=\sum_{j=1}^kz_j^\uparrow.\] Thus,
$\textup{diag}_{\mathbf{M}}(A)=z\prec^w \delta(A)$.

\end{proof}

The next result is the symplectic Ky Fan's theorem for generalized means, and it immediately follows from Theorems \ref{theorem6.2} and \ref{equiv}.

\begin{corollary}
Let $\mathbf M$ be a mean such that $\sqrt{ab}\leq \mathbf M(a,b)$ for all $a,b>0$. If $A\in M_{2n}$ is positive definite with symplectic eigenvalues $\delta_1,\ldots, \delta_n$, then $\sum_{j=1}^k\delta_j^\uparrow=\min\{\sum_{j=1}^k \mathbf{M}(b_{jj},b_{k+j,k+j}): [b_{ij}]_{i,j=1}^{2k}\in {\cal S}_k(A)\}$.

\end{corollary}
\section*{Declaration of Competing Interest}
There is no competing interest.

\end{document}